\documentclass[11pt,twoside]{article}
\usepackage{amsmath, amsthm, amscd, amsfonts, amssymb, graphicx, color}

\date{}

\begin{document}

\centerline{}

\centerline {\Large{\bf Uniform Boundedness Principle and Hahn-Banach Theorem }}
\centerline{\Large{\bf for b-linear functional related to linear 2-normed space }}
%% My definition
\newcommand{\mvec}[1]{\mbox{\bfseries\itshape #1}}
\centerline{}
\centerline{\textbf{Prasenjit Ghosh}}
\centerline{Department of Pure Mathematics, University of Calcutta,}
\centerline{35, Ballygunge Circular Road, Kolkata, 700019, West Bengal, India}
\centerline{e-mail: prasenjitpuremath@gmail.com}
\centerline{\textbf{Sanjay Roy}}
\centerline{Department of Mathematics, Uluberia College,}
\centerline{Uluberia, Howrah, 711315,  West Bengal, India}
\centerline{e-mail: sanjaypuremath@gmail.com}
\centerline{\textbf{T. K. Samanta}}
\centerline{Department of Mathematics, Uluberia College,}
\centerline{Uluberia, Howrah, 711315,  West Bengal, India}
\centerline{e-mail: mumpu$_{-}$tapas5@yahoo.co.in}

\newtheorem{Theorem}{\quad Theorem}[section]

\newtheorem{definition}[Theorem]{\quad Definition}

\newtheorem{theorem}[Theorem]{\quad Theorem}

\newtheorem{remark}[Theorem]{\quad Remark}

\newtheorem{corollary}[Theorem]{\quad Corollary}

\newtheorem{note}[Theorem]{\quad Note}

\newtheorem{lemma}[Theorem]{\quad Lemma}

\newtheorem{example}[Theorem]{\quad Example}

\newtheorem{result}[Theorem]{\quad Result}
\newtheorem{conclusion}[Theorem]{\quad Conclusion}

\newtheorem{proposition}[Theorem]{\quad Proposition}

\begin{abstract}
\textbf{\emph{In this paper,\;we will see that the Cartesian product of two 2-Banach spaces is also 2-Banach space and discuss some properties of closed linear operator in linear 2-normed space\,.\;We also describe the concept of different types of continuity of b-linear functional and derive the Uniform Boundedness Principle and Hahn-Banach extension theorem for b-linear functionals in the case of linear 2-normed spaces\,.\;We also introduce the notion of weak\,*\,convergence for the sequence of bounded b-linear functionals relative to linear 2-normed space\,.}}
\end{abstract}

\\
\\

%=====================================
\section{Introduction}
%=====================================

\smallskip\hspace{.6 cm}  The Uniform boundedness principle is one of the most useful results in functional analysis which was obtained by S. Banach and H. Steinhaus in \,$1927$\; and it is also familiar as Banach-Steinhaus Theorem\,.\;The Uniform boundedness principle tells us that if a sequence of bounded linear operators \,$T_{\,n}\, \,\in\, B\,(\,X \,,\, Y\,)$, where \,$X$\; is a Banach space and \,$Y$\; a normed space, is pointwise bounded, then the sequence \,$\{\,T_{\,n}\,\}$\; is uniformly bounded\,.\\
The Hahn-Banach theorem is another useful and important theorem in functional analysis and it is frequently applied in other branches of mathematics viz\,.\,, algebra, geometry, optimization, partial differential equation and so on\,.\;In fact, in this theorem a bounded linear functional defined on a subspace can be extended into the entire space\,.\\
The idea of linear 2-normed space was first introduced by S. Gahler (\cite{Gahler}) and thereafter the geometric structure of linear 2-normed spaces was developed by the great mathematicians like A. White, Y. J. Cho, R. W. Freese, S. C. Gupta and others \cite{White, Gupta, Freese}\,.\;In recent times, some important results in classical normed spaces have been proved into 2-norm setting by many researchers\,. \\In this paper, we will see that in the Cartesian product \,$X \,\times\, Y $, we can induce a 2-norm using the 2-norms of \,$X$\, and \,$Y$\, and then we describe the concept of different types of continuity of b-linear functional in the case of linear 2-normed spaces and establish some results related to such types of continuity\,.\;In this paper, we are going to construct the uniform boundedness principle and Hahn-Banach extension theorem for a bounded b-linear functional defined on a 2-Banach space\,.\;Moreover we introduce a notion of weak\,*\,convergence of the sequence of bounded b-linear functionals in linear 2-normed spaces\,.

%=====================================
\section{Preliminaries}
%=====================================

\smallskip\hspace{.6 cm}
\begin{definition}(\cite{Gahler})
Let \,$X$\, be a linear space of dimension greater than \,$1$\, over the field \,$ \mathbb{K} $\,, where \, 
$ \mathbb{K} $\, is the real or complex numbers field and \,$\|\, \cdot \;,\; \cdot\,\|$\; be a \,$\mathbb{K}$-valued function defined on \,$X \,\times\, X$\; satisfying the following conditions:
\begin{itemize}
\item[(N1)] $\|\,x \;,\; y \,\| \;\,=\,\; 0 $\; if and only if \,$x \;,\; y $\; are linearly dependent,
\item[(N2)] $\|\,x \;,\; y\,\| \;\,=\,\; \|\,y \;,\; x \,\|$\,,
\item[(N3)] $\|\,\alpha \,x \;,\; y\,\|\;\,=\,\;|\,\alpha\,| \;\| \,x \;,\; y\,\|\; \;\forall \;\;\alpha \;\in\; \mathbb{K}$\,,
\item[(N4)] $\|\,x \;,\; y \;+\; z\,\| \,\;\leq\,\; \|\,x \;,\; y\,\| \;+\; \|\,x \;,\; z\,\|$\,.
\end{itemize}
Then \,$\,\|\,\cdot \;,\; \cdot\,\|\,$\, is called a 2-norm on \,$X$\, and the pair \,$\left(\,X \;,\; \|\,\cdot \;,\; \cdot\,\| \,\right)$\; is called a linear 2-normed space\,.\;The non-negativity condition of 2-norm can be obtained by using \,$(\,N3\,) \;\&\; (\,N4\,)$\,.
\end{definition}

\begin{definition}(\cite{Freese})
Let \,$X$\; be a linear 2-normed space\,.\;A sequence \,$\{\,x_{\,n}\,\}$\; in \,$X$\, is said to be convergent to some \,$x \,\in\, X$\; if 

\[\lim\limits_{n \to \infty}\,\|\,x_{\,n} \;-\; x \;,\; y\,\| \;\;=\;\; 0 \]
for every \;$y \;\in\; X$\; and it is called Cauchy sequence if 
\[\lim\limits_{n \,,\,m \to \infty}\,\|\,x_{\,n} \;-\; x_{\,m} \;,\; z\,\| \;\;=\;\; 0 \]
for every \;$z \;\in\; X$\,.\;$X$\, is said to be complete if every Cauchy sequence in this space is convergent in \,$X$\,.\;A linear 2-normed space is called 2-Banach space if it is complete\,.
\end{definition}

\begin{definition}(\cite{Riyas})
Define the following open and closed ball in linear 2-normed space \,$X$\,:
\[ B_{\,e}\,(\,a \;,\; \delta\,) \;\,=\,\; \left\{\;x \;\in\; X \;:\; \|\,x \;-\; a \;,\; e\,\| \;<\; \delta \;\right\} \]
\[ B_{\,e}\,[\,a \;,\; \delta\,] \;\,=\,\; \left\{\;x \;\in\; X \;:\; \|\,x \;-\; a \;,\; e\,\| \;\,\leq\,\; \delta \;\right\}\;,\]
where \,$a \,,\, e \;\in\; X$\; and \,$\delta$\; be a positive number\,.
\end{definition}

\begin{definition}(\cite{Riyas})
A subset \,$G$\; of a linear 2-normed space \,$X$\; is said to be open in \,$X$\; if for all \,$ a \,\in\, G $, there exists \,$e \,\in\, X $\; and \,$\delta \,>\, 0$\; such that \,$ B_{\,e}\,(\,a \;,\; \delta\,) \;\subseteq\; G$\,.
\end{definition}

\begin{definition}(\cite{Ravindran})
Let \,$X$\; be a linear 2-normed space and \,$A \,\subseteq\, X$\,.\;Then a point \,$ a \,\in\, A $\; is said to an interior point of \,$A$\; if there exists \,$e \,\in\, X$\; and \,$\delta \,>\, 0$\; such that \,$B_{e}\,(\,a \;,\; \delta\,) \,\subseteq\, A$\,.
\end{definition}

\begin{definition}(\cite{Riyas}).
Let \,$X$\; be a linear 2-normed space\,.\;Then \,$ G \,\subseteq\, X$\; is said to be dense in \,$X$\; if \;$V \,\cap\, G \,\neq\, \phi$\; for every open set \,$V$\; in \,$X$\,.
\end{definition}

\begin{definition}(\cite{Ravindran})
Let \,$X$\; be a linear 2-normed space and \,$A \,\subseteq\, X$\,.\;Then the closure of \,$A$\, is denoted by \,$ \overline{A} $\, and defined as\,, 
\[ \{\; x \;\in\; X \;|\; \;\exists\; \;\{\,x_{\,n}\,\} \;\in\; A \;\;\textit{with}\;  \lim\limits_{n \,\to\, \infty} x_{\,n} \;=\; x \;\}\;.\]
The set \,$ A $\; is said to be closed if \,$ A \;=\; \overline{A}$. 
\end{definition}

\begin{theorem}\,(Baire's theorem for 2-Banach spaces)\label{th1}(\cite{Pilakkat, Riyas})
Let \,$X$\, be a 2-Banach space\,.\;Then the intersection of a countable number of dense open subsets of \,$X$\, is dense in \,$X$\,.
\end{theorem}

\begin{definition}(\cite{Pilakkat})
Let \,$X$\, and \,$Y$\, be two linear 2-normed spaces over the field \,$\mathbb{K}$\,.\;Then a linear operator \,$T \,:\, X \,\to\, Y$\; is said to be closed if for every \,$\left\{\,x_{\,n}\,\right\}$\; in \,$X$\, with \,$x_{\,n} \,\to\, x $\; and \;$T\,(\,x_{\,n}\,) \,\to\, y$\; in \,$Y$\, we have \;$x \;\in\; X$\; and \;$T\,(\,x\,) \,=\, y$\,.
\end{definition}

\begin{definition}(\cite{Riyas})
Let \,$X$\, and \,$Y$\, be two linear 2-normed spaces over \,$\mathbb{R}$\, and \,$T \,:\, X \,\to\, Y$\; be a linear operator\,.\;The operator \,$T$\, is said to be sequentially continuous at \,$x \,\in\, X$\, if and only if for every sequence \,$\left\{\,x_{\,n}\,\right\}$\; in \,$X$\, that converges to \,$x$, the sequence \,$\left\{\,T\,(\,x_{\,n}\,)\,\right\}$\; converges to \;$T\,(\,x\,)$\,.
\end{definition}

\begin{theorem}\label{th2}(\cite{Riyas})
Let \,$X$\, and \,$Y$\, be two linear 2-normed spaces over \,$\mathbb{R}$\,.\;If \,$X$\, is finite dimensional, then every linear operator \,$T \,:\, X \,\to\, Y$\; is sequentially continuous\,.
\end{theorem}

\begin{definition}(\cite{Harikrishnan})
Let \,$\left(\,X \;,\; \|\; \cdot \;,\; \cdot \;\|\,\right)$\; be a linear 2-normed space over the field \,$\mathbb{K}$\; with \,$b \,\in\, X$\; be fixed  and \,$\left <\,b\,\right >$\; is the subspace of \,$X$\, generated by \,$b$\,.\;Let \,$W$\, be a subspace of \,$X$, then a mapping \,$T \,:\, W \,\times\, \left <\,b\,\right > \,\to\, \mathbb{K}$\; is called a b-linear functional on \,$W \,\times\, \left <\,b\,\right >$, if the following two hold:
\begin{itemize}
\item[(1)] \;$T\,(\;x \;+\; y \;,\; b\;) \;=\; T\,(\;x \;,\; b\;) \;+\; T\,(\;y \;,\; b\;)\; \;\forall\; x\;,\; y \;\in\; W$\,.
\item[(2)] \;$T\,(\;k\,x \;,\; b\;) \;=\; k\; T\,(\;x \;,\; b\;)\; \;\forall\; k \;\in\; \mathbb{K}$\,.
\end{itemize}
A \,b-linear functional \,$T \,:\, W \,\times\, \left <\,b\,\right > \,\to\, \mathbb{K}$\; is said to be bounded if there exists a real number \,$M \,>\, 0$\, such that 
\[|\,T\,(\;x \;,\; b\;)\,| \;\leq\; M \;\|\;x \;,\; b\;\|\; \;\forall\; x \;\in\; W\;.\]\;Now we can define the norm of the b-linear functional \;$T \,:\, W \,\times\, \left <\,b\,\right > \,\to\, \mathbb{K}$\; as \[\|\, T \,\| \;\,=\,\; \inf \left\{\,M \;>\; 0 \;:\; |\, T\,(\,x \;,\; b\,) \,| \;\leq\; M\, \|\,x \;,\; b\,\|\;\;  \;\forall\;x \;\in\; W\,\right\}\;.\]
Then one can easily verified that,
\[ \|\,T\,\| \;\,=\,\; \sup \{\,|\,T\,(\,x \;,\; b\,)\,| \;:\; \|\,x \;,\; b\,\| \;\leq\; 1\,\} \]
\[ \|\,T\,\| \;\,=\,\; \sup \{\,|\,T\,(\,x \;,\; b\,)\,| \;:\; \|\,x \;,\; b\,\| \;\,=\,\; 1\,\} \] 
\[\hspace{.3cm}  \|\,T\,\| \;\,=\,\; \sup \left \{\,\dfrac{|\,T\,(\,x \;,\; b\,)\,|}{\|\,x \;,\; b\,\|} \;:\; \|\,x \;,\; b\,\| \;\neq\; 0\,\right \} \] and then 
\[ |\, T\,(\,x \;,\; b\,) \,| \;\leq\; \|\,T\,\| \; \|\,x \;,\; b\,\| \; \;\forall\; x \;\in\; W\;.\]
\end{definition}
Let \;$X_{\,b}^{\,\ast}$\; denote the Banach space of all bounded b-linear functionals defined on \;$X \,\times\, \left <\,b\,\right > $\,.

%=====================================
\section{Analogous Results in the classical normed spaces to 2-normed spaces}
%=====================================

\begin{theorem}
Let \,$\left(\,X \;,\; \|\; \cdot \;,\; \cdot \;\|_{X} \,\right)$\; and \;$\left(\,Y \;,\; \|\; \cdot \;,\; \cdot \;\|_{Y} \,\right)$\; be two linear 2-normed linear spaces over the field \,$\mathbb{K}$\,, Then in the Cartesian product \,$X \,\times\, Y$, we can induce a 2-norm \;$\|\,\cdot \;,\; \cdot\,\|$\; using the 2-norms of \,$X$\, and \,$Y$\,.\;Furthermore, if \,$X$\, and \,$Y$\, are 2-Banach spaces then \,$X \,\times\, Y$\; is also 2-Banach space\,.
\end{theorem}

\begin{proof}
Define a function \,$\|\, \cdot \;,\; \cdot \;\|\; \,:\, \left(\,X \,\times\, Y\,\right) \,\times\, \left(\,X \,\times\, Y\,\right) \,\to\, \mathbb{R}$\; by,
\[ \left\|\,(\,x_{\,1} \,,\, y_{\,1}\,) \,,\, (\,x_{\,2} \,,\, y_{\,2}\,)\,\right\| \,=\, \left\|\,x_{\,1} \,,\, x_{\,2}\,\right\|_{X} \,+\, \left\|\,y_{\,1} \,,\, y_{\,2}\,\right\|_{Y}\]
for all \,$(\,x_{\,1} \,,\, y_{\,1}\,) \,,\, (\,x_{\,2} \,,\, y_{\,2}\,) \,\in\, (\,X \,\times\, Y\,)$\,.\;We now verify that this function is a 2-norm on \,$X \,\times\, Y$\,.
\begin{itemize}
\item[(\,N\,1\,)]\;\; Suppose
\[\left\|\, (\,x_{\,1} \,,\, y_{\,1}\,) \;,\; (\,x_{\,2} \,,\, y_{\,2}\,) \,\right\| \,=\, 0 \; \;\forall\; (\,x_{\,1} \,,\, y_{\,1}\,) \,,\, (\,x_{\,2} \,,\, y_{\,2}\,) \,\in\, (\,X \,\times\, Y\,)\]
\[\Leftrightarrow \left\|\,x_{\,1} \;,\; x_{\,2}\,\right\|_{X} \;+\; \left\|\, y_{\,1} \;,\; y_{\,2}\,\right\|_{Y} \,=\, 0 \hspace{6.4cm} \]
\[ \hspace{.2cm}\;\Leftrightarrow\; \left\|\,x_{\,1} \;,\; x_{\,2}\,\right\|_{X} \;\,=\,\; 0 \;,\; \left\|\, y_{\,1} \;,\; y_{\,2}\,\right\|_{Y} \;=\; 0\; \;\text{for}\; \;x_{\,1} \;,\; x_{\,2} \;\in\; X \;\;\&\;\; y_{\,1} \;,\; y_{\,2} \;\in\; Y\]
$ {\hspace{.1cm}} \;\Leftrightarrow\; \left\{\,x_{\,1} \;,\; x_{\,2}\,\right\}$\; and \;$\left\{\,y_{\,1} \;,\; y_{\,2}\,\right\}$\; are linearly dependent in \,$X \;\&\; Y$\\
$ {\hspace{.1cm}} \;\Leftrightarrow\; (\,x_{\,1} \;,\; y_{\,1}\,) \;,\; (\,x_{\,2} \;,\; y_{\,2}\,) $\; are linearly dependent in \;$X \;\times\; Y$\,.
\item[(\,N\,2\,)]\;\; Now, 
\[\left\|\,(\,x_{\,1} \,,\, y_{\,1}\,) \;,\; (\,x_{\,2} \,,\, y_{\,2}\,)\,\right\| \;\,=\,\; \left\|\, x_{\,1} \;,\; x_{\,2} \,\right\|_{X} \;+\; \left\|\, y_{\,1} \;,\; y_{\,2}\,\right\|_{Y} \]
\[ \hspace{4.5cm}\;\,=\,\; \|\,x_{\,2} \;,\; x_{\,1} \,\|_{\,X} \;+\; \|\, y_{\,2} \;,\; y_{\,1}\,\|_{\,Y} \]
\[ \hspace{4.1cm}\;\,=\,\; \left\|\,(\,x_{\,2} \;,\; y_{\,2}\,) \;,\; (\,x_{\,1} \;,\; y_{\,1}\,)\,\right\|\]
for all \,$(\,x_{\,1} \;,\; y_{\,1}\,) \;,\; (\,x_{\,2} \;,\; y_{\,2}\,) \;\in\; (\,X \,\times\, Y\,)$\,.
\item[(\,N\,3\,)]\;\; Let \,$\alpha \;\in\; \mathbb{K}$, then
\[\left\|\, \;\alpha\; (\,x_{\,1} \;,\; y_{\,1}\,) \;,\;(\,x_{\,2} \;,\; y_{\,2}\,) \,\right\| \;\,=\,\; \left\|\,(\,\; \alpha\; x_{\,1} \;,\;  \;\alpha\; y_{\,1}\,) \;,\; (\,x_{\,2} \;,\; y_{\,2}\,) \,\right\| \]
\[\hspace{5.3cm} \;\,=\,\; \left\|\, \;\alpha\; x_{\,1} \;,\; x_{\,2}\,\right\|_{X} \;+\; \left\|\;\alpha\; y_{\,1} \;,\; y_{\,2}\,\right\|_{Y}\]
\[\hspace{6cm} \;\,=\,\; |\;\alpha\;|\, \left\|\,x_{\,1} \;,\; x_{\,2}\,\right\|_{X} \;+\; |\;\alpha\;|\,\left\|\,y_{\,1} \;,\; y_{\,2}\,\right\|_{Y}\]
\[\hspace{5.5cm} \;\,=\,\; |\;\alpha\;|\,\left(\,\left\|\,x_{\,1} \;,\; x_{\,2}\,\right\|_{X} \;+\; \left\|\,y_{\,1} \;,\; y_{\,2}\,\right\|_{Y}\, \right)\]
\[\hspace{4.7cm} \;\,=\,\; |\, \;\alpha\;|\,\left\|\,(\,x_{\,1} \;,\; y_{\,1}\,) \;,\; (\,x_{\,2} \;,\; y_{\,2}\,)\,\right\|\]
${\hspace{1.5cm}}$ for every \;$(\,x_{\,1} \;,\; y_{\,1}\,) \;,\; (\,x_{\,2} \;,\; y_{\,2}\,) \;\in\; (\,X \,\times\, Y\,)$\,.
\item[(\,N\,4\,)]\;\;
\[\left\|\,(\,x_{\,1} \,,\, y_{\,1}\,) \,+\, (\,x_{\,2} \,,\, y_{\,2}\,) \;,\; (\,x_{\,3} \,,\, y_{\,3}\,)\,\right\|  \,=\, \left\|\,(\,x_{\,1} \,+\, x_{\,2} \;,\; y_{\,1} \,+\, y_{\,2}\,) \;,\; (\,x_{\,3} \;,\; y_{\,3}\,)\,\right\| \]
\[ \;\,=\,\; \left\|\,x_{\,1} \;+\; x_{\,2} \;,\; x_{\,3}\,\right\|_{X} \;+\; \left\|\,y_{\,1} \;+\; y_{\,2} \;,\; y_{\,3}\,\right\|_{Y}\hspace{.3cm}\]
\[ {\hspace{2.5cm}} \;\leq\; \left\|\,x_{\,1} \;,\; x_{\,3}\,\right\|_{X} \;+\; \left\|\,x_{\,2} \;,\; x_{\,3}\,\right\|_{X} \;+\; \left\|\,y_{\,1} \;,\; y_{\,3}\,\right\|_{Y} \;+\; \left\|\,y_{\,2} \;,\; y_{\,3}\,\right\|_{Y}\]
\[{\hspace{2.5cm}} \,\;=\,\; \left\|\,x_{\,1} \;,\; x_{\,3}\,\right\|_{X} \;+\; \left\|\,y_{\,1} \;,\; y_{\,3}\,\right\|_{Y} \;+\; \left\|\,x_{\,2} \;,\; x_{\,3}\,\right\|_{X} \;+\; \left\|\,y_{\,2} \;,\; y_{\,3}\,\right\|_{Y}\]
\[ {\hspace{2cm}} \;\,=\,\; \left\|\,(\,x_{\,1} \;,\; y_{\,1}\,) \;,\; (\,x_{\,3} \;,\; y_{\,3}\,)\,\right\| \;+\; \left\|\,(\,x_{\,2} \;,\; y_{\,2}\,) \;,\; (\,x_{\,3} \;,\; y_{\,3}\,)\,\right\|\]
for every \;$(\,x_{\,1} \;,\; y_{\,1}\,) \;,\; (\,x_{\,2} \;,\; y_{\,2}\,) \;,\; (\,x_{\,3} \;,\; y_{\,3}\,) \;\in\; \left(\,X \,\times\, Y\,\right)$\,.
\end{itemize}
Thus, \,$\left(\,X \;\times\; Y \;,\; \|\,\cdot \;,\; \cdot \;\|\; \right)$\; becomes a linear 2-normed space\,.
\end{proof}

\textit{Second part}\,: \;Let \,$\left\{\,(\,x_{\,n} \;,\; y_{\,n}\,)\,\right\}$\; be a Cauchy sequence in \;$X \,\times\, Y$\,.\;Then 
\[ \lim\limits_{n \,,\, m \,\to\, \infty} \left\|\,(\,x_{\,n} \;,\; y_{\,n}\,) \;-\; (\,x_{\,m} \;,\; y_{\,m}\,) \;,\; (\,z \;,\; t\,)\,\right\| \;=\; 0 \;\; \;\forall\; (\,z \;,\; t\,) \;\in\; X \,\times\, Y \]
\[ \Rightarrow\; \lim\limits_{n \,,\, m \,\to\, \infty} \left\|\,(\,x_{\,n} \;-\; x_{\,m} \;,\; y_{\,n} \;-\; y_{\,m}\,) \;,\; (\,z \;,\; t\,) \;\right\|  \;=\; 0  \;  \;\forall\; (\,z \,,\, t\,) \,\in\, X \,\times\, Y\]
\[\Rightarrow\; \lim\limits_{n \,,\, m \,\to\, \infty} \;\left(\,\left\|\,x_{\,n} \;-\; x_{\,m} \;,\; z\,\right\|_{X} \;+\; \left\|\,y_{\,n} \;-\; y_{\,m} \;,\; t\,\right\|_{Y}\,\right)\; \;=\; 0\;.\hspace{2cm}\]
Therefore,
\[\lim\limits_{ n \,,\,m \,\to\, \infty} \left\|\,x_{\,n} \,-\, x_{\,m} \,,\, z\,\right\|_{X} \,=\, 0 \; \;\forall\; z \,\in\, X\]and
\[\lim\limits_{ n \,,\,m  \,\to\, \infty}\,  \left\|\,y_{\,n} \,-\, y_{\,m} \,,\, t\,\right\|_{Y} \,=\, 0 \; \;\forall\; t \,\in\, Y\;.\]
This shows that \;$\left\{\,x_{\,n}\,\right\}$\; and \;$\left\{\,y_{\,n}\,\right\}$\; are Cauchy sequences in \,$X$\; and \,$Y$, respectively\,.\;Since \,$X$\, and \,$Y$\, are 2-Banach spaces, So there exists points \,$x \;\in\; X$\, and \,$y \;\in\; Y$\, such that \;$x_{\,n} \,\to\, x$\; in \,$X$\, and \;$y_{\,n} \,\to\, y$\; in \,$Y$\, and hence \;$\left(\,x_{\,n} \;,\; y_{\,n}\,\right) \,\to\, (\,x \;,\; y\,)$\; in \;$X \,\times\, Y$\,.\;Therefore, \,$X \,\times\, Y$\; is 2-Banach space\,.

\begin{theorem}
Let \,$X$\, and \,$Y$\, be two linear 2-normed spaces over the field \,$\mathbb{K}$\, and \,$D$\, be a subspace of \,$X$\,.\;Then the linear operator \,$T \,:\, D \,\to\, Y$\; is closed if and only if its Graph is a closed subspace of \,$X \,\times\, Y$\,.
\end{theorem}

\begin{proof}
First we suppose that \,$T \,:\, D \,\to\, Y$\; is closed operator, that is the relation \,$x_{\,n} \,\in\, D \;,\; x_{\,n} \,\to\, x\, \;\text{in}\, \,X \;,\; T\,x_{\,n} \,\to\, y\; \;\text{in}\; \;Y$\; implies that \,$x \,\in\, D$\, and \,$T\,x  \,=\, y$\,.\;We shall prove that the graph \;$G_{\,T} \,=\, \left\{\,(\,x \,,\, T\,x\,) \,:\, x \,\in\, D\,\right\}$\; is closed in linear 2-normed space \;$X \,\times\, Y$\,.\;Let \,$\left\{\,(\,x_{\,n} \,,\, T\,x_{\,n}\,)\,\right\} \,\subseteq\, G_{\,T} \;,\; x_{\,n} \,\in\, D$\; and \,$\left(\,x_{\,n} \,,\, T\,x_{\,n}\,\right) \,\to\, (\,x \,,\, y\,)$\; as \,$n \,\to\, \infty$\,.\;Therefore,
\[ \lim\limits_{n \to \infty} \left\|\,\left(\,x_{\,n} \,,\, T\,x_{\,n}\,\right) \,-\, (\,x \,,\, y\,) \;,\; (\,z \,,\, t\,) \,\right\| \,=\, 0\; \;\;\forall\; (\,z \,,\, t) \,\in\, X \,\times\, Y \]
\[ \Rightarrow\; \lim\limits_{n \to \infty} \left\|\,\left(\,x_{\,n} \,-\, x \,,\, T\,x_{\,n} \,-\, y\,\right) \;,\; (\,z \,,\, t\,) \,\right\| \,=\, 0\; \;\;\forall\; (\,z \,,\, t) \,\in\, X \,\times\, Y \]
\[ {\hspace{1.2cm}} \Rightarrow\; \lim\limits_{n \to \infty} \left(\, \left\|\, x_{\,n} \,-\, x \;,\; z \,\right\|_{X} \,+\, \left\|\, T\,x_{\,n} \,-\, y \;,\; t \,\right\|_{Y}\,\right) \,=\, 0\;\; \;\forall\; (\,z \,,\, t) \,\in\, X \,\times\, Y \;.\]Thus,
\[\lim\limits_{n \to \infty} \left\|\, x_{\,n} \,-\, x \;,\; z \,\right\|_{X} \,=\, 0\;\; \;\forall\; z \,\in\, X\]and
\[\lim\limits_{n \to \infty} \left\|\,T\,x_{\,n} \,-\, y \;,\; t \,\right\|_{\,Y} \,=\, 0\;\; \;\forall\; t \,\in\, Y\;.\] 
This shows that \,$ x_{\,n} \,\to\, x $\; and \;$ T\,x_{\,n} \,\to\, y $\; as \,$n \,\to\, \infty $\,.\;Since \,$T$\, is closed operator, We have \,$x \,\in\, D $\, and \,$ T\,x \,=\, y $\; and therefore \,$(\, x \,,\, y \,) \,=\, (\, x \,,\, T \,x \,) \,\in\, G_{\,T}$\,.\;Hence, \,$ G_{\,T} $\; is closed subspace of linear 2-normed space \,$X \,\times\, Y$\,.\\\textit{Conversely}\,, Suppose \,$G_{\,T}$\, is closed subspace of linear 2-normed space \,$X \,\times\, Y $\,.\;To prove \,$T$\; is closed operator, we consider \,$x_{\,n} \,\to\, x \;,\; x_{\,n} \,\in\, D$\; and \,$T \,x_{\,n} \,\to\, y $\,.\;Now,
\[\left\|\,\left(\, x_{\,n} \,,\, T\,x_{\,n}\,\right) \,-\, (\,x \,,\, y \,) \;,\; (\, z \,,\, t \,)\,\right\| \,=\, \left\|\,\left(\, x_{\,n} \,-\, x \;,\; T\,x_{\,n} \,-\, y \,\right) \;,\; (\, z \,,\, t \,)\,\right\|\,\]
\begin{equation}\label{eq1}
\;=\; \left\|\, x_{\,n} \;-\; x \;,\; z \,\right\|_{X} \;+\; \left\|\, T\,x_{\,n} \;-\; y \;,\; t \,\right\|_{Y}\;. 
\end{equation} 
Since \,$x_{\,n} \,\to\, x $\; and \,$ T\,x_{\,n} \,\to\, y $\; as \,$n \,\to\, \infty$\,, then
\[\lim\limits_{n \to \infty} \left\|\, x_{\,n} \,-\, x \,,\, z \,\right\|_{X} \,=\, 0\; \;\forall\; z \,\in\, X\]and
\[\lim\limits_{n \to \infty} \left\|\, T\,x_{\,n} \,-\, y \,,\, t \,\right\|_{Y} \,=\, 0\; \;\forall\; t \,\in\, Y\;.\]
So by (\ref{eq1}), 
\[\lim\limits_{n \to \infty} \left\|\,\left(\,x_{\,n} \,,\, T\,x_{\,n} \,\right) \,-\, (\, x \,,\, y \,) \;,\; (\, z \,,\, t\,) \,\right\| \,=\, 0\;\; \;\forall\; (\, z \,,\, t \,) \,\in\, X \,\times\, Y\;.\] 
This shows that \,$(\, x_{\,n} \,,\, T\,x_{\,n} \,) \,\to\, (\, x \,,\, y \,)$\, as \;$ n \,\to\, \;\infty $\,.\;Since \,$G_{\,T}$\, is closed subspace of linear 2-normed space \;$X \,\times\, Y $\,, it follows that \,$(\, x \,,\, y \,) \,\in\, G_{\,T}$\,, that is\,, \,$ x \,\in\, D$\; and \;$y \,=\, T\,x $\,. Hence, \,$T$\, is closed linear operator\,.
\end{proof}

%=====================================
\section{Some properties related to b-linear functional}
%=====================================

In this section we define different types of continuity of b-linear functional and give some characterizations between them in linear 2-normed spaces\,.

\begin{theorem}\label{th3}
Let \,$X$\; be a linear 2-normed space\,.\;Then 
\[ \left |\, \|\, x \,,\, z \,\| \,-\, \|\, y \,,\, z \,\| \,\right |\;\leq\; \|\, x \,-\, y \;,\; z \,\|\; \;\forall\; x \;,\; y \;,\; z \;\in\; X \;.\]
\begin{proof}
Take \,$x \;,\; y \;,\; z \;\in\; X $\,. Then 
\[\|\, x \,,\, z \,\| \;\,=\,\; \|\, x \;-\; y \;+\; y \;,\; z \,\| \;\leq\; \|\, x \;-\; y \;,\; z \,\| \;+\; \|\, y \;,\; z \,\|\]
\[ \;\Rightarrow\; \| x \;,\; z \,\| \;-\; \|\, y \;,\; z \,\| \;\leq\; \|\, x \;-\; y \;,\; z \,\|\;. \hspace{.5cm}\]
Also, interchanging \,$x$\, and \,$y$\,, we get that 
\[\;-\; \left(\; \| x \;,\; z \,\| \;-\; \|\, y \;,\; z \,\|\;\right) \;\leq\; \|\, y \;-\; x \;,\; z \,\| \;\,=\,\; \|\, x \;-\; y \;,\; z \,\|\;.\]
Combining the above two inequality the result follows\,.
\end{proof}
\end{theorem}

\begin{theorem}\label{th4}
Let \,$T$\, be a bounded b-linear functional on \,$ X \,\times\, \left <\,b\,\right > $, where \,$\left <\,b\,\right > $\; is the subspace of the linear 2-normed space \,$X$\, generated by a fixed \,$b \,\in\, X$\,.\;Then 
\[\left |\, T\,(\, x \,,\, b \,) \,-\, T\,(\, y \,,\, b \,) \,\right |\, \;\leq\; \|\, T \,\| \,\|\, x \,-\, y \;,\; b \,\|\;\; \;\forall\; x \;,\; y \,\in\, X \;.\]
\end{theorem}

\begin{proof}
For each \,$x \,,\, y \,\in\, X$ 
\[|\, T\,(\, x \,,\, b \,) \,-\, T\,(\, y \,,\, b \,) \,| \,=\, \left|\, T\,(\, x \,,\, b \,) \,+\, T\,(\,-\, y \,,\, b \,) \,\right| \hspace{2cm}\]
\[{\hspace{4cm}} \,=\, \left|\, T\,(\, x \,-\, y \,,\, b \,)\,\right| \,\leq\, \|\, T \,\| \,\left\|\, x \,-\, y \,,\, b \,\right\|\;.\]
\end{proof}

\begin{definition}
Let \,$T$\, be a b-linear functional defined on \,$ X \,\times\, \left <\,b\,\right > $\,.\;Then \,$T$\, is said to be b-sequentially continuous at \,$x \,\in\, X $\; if for every sequence \,$\left\{\,x_{\,n}\,\right\}$\; converging to \,$x$\, in \,$X$, we have \,$\left\{\,T\,(\, x_{\,n} \,,\, b \,)\,\right\}$\; converging to \,$T\,(\, x \,,\, b \,)$\; in \,$\mathbb{K}$\,.
\end{definition}

\begin{theorem}\label{th5}
Let \,$X$\, be a linear 2-normed space over \,$\mathbb{K}$\, and \,$b \,\in\, X $\, be fixed\,.\;Then every bounded b-linear functional defined on \,$X \,\times\, \left <\,b\,\right > $\; is b-sequentially continuous\,.
\end{theorem}

\begin{proof}:
Let \,$T$\, be a bounded b-linear functional on \,$X \,\times\, \left <\,b\,\right >$\, and \,$\{\,x_{\,n}\,\}$\; be a sequence converging to \,$x$\, in \,$X$\,.\;Then, 
\[\lim\limits_{n \to \infty}\, \left \|\, x_{\,n} \,-\, x \;,\; z \,\right \| \,=\, 0 \; \;\forall\; z \;\in\; X \;,\] and for particular \,$z \,=\, b$, we can write, \,$\lim\limits_{n \to \infty}\, \left\|\, x_{\,n} \,-\, x \;,\; b \,\right\| \;\,=\,0$\,.\;Now, using Theorem (\ref{th4}), by putting \,$x \,=\, x_{\,n}$\; and \,$y \,=\, x$\,, we can write
\[ \left|\, T\,(\, x_{\,n} \,,\, b \,) \,-\, T\,(\, x \,,\, b \,)\,\right| \;\leq\; \|\, T \,\|\, \left\|\, x_{\,n} \,-\, x \;,\; b \,\right\| \]
\[\;\Rightarrow\; \lim\limits_{n \to \infty}\,  \left|\, T\,(\, x_{\,n} \,,\, b \,) \,-\, T\,(\, x \,,\, b \,)\,\right| \;\leq\; \|\, T \,\|\, \lim\limits_{ n \to \infty} \,\left\|\, x_{\,n} \,-\, x \;,\; b \,\right\| \]
\[\Rightarrow\, \lim\limits_{n \to \infty}\,  \left|\, T\,(\, x_{\,n} \,,\, b \,) \,-\, T\,(\, x \,,\, b \,)\,\right| \;=\; 0\;. \hspace{3.8cm}\] 
Therefore, \,$\left\{\,T\,(\, x_{\,n} \,,\, b \,)\,\right\}$\; converging to \,$T\,(\,x \,,\, b \,)$\; in \,$\mathbb{K}$\,.\;Hence, \,$T$\, is b-sequentially continuous\,.
\end{proof}

\begin{definition}
Let \,$X$\, be a linear 2-normed space and \,$b \,\in\, X$\, be fixed\,.\;Then a b-linear functional \,$T \,:\, X \,\times\, \left <\,b\,\right > \,\to\, \mathbb{K}$\; is said to be continuous at \,$x_{\,0} \;\in\; X $\; if for any open ball \,$ B\,\left(\, T\,(\, x_{\,0} \,,\, b \,) \;,\; \epsilon\,\,\right)$\; in \,$\mathbb{K}$, there exist an open ball $B_{\,e}\,(\, x_{\,0} \,,\, \delta \,)$\; in \,$X$\, such that
\[T\,\left(\, B_{\,e}\,(\, x_{\,0} \,,\, \delta \,) \,,\, b \,\right) \,\subseteq\, B\, \left(\; T\,(\, x_{\,0} \,,\, b \,) \,,\, \epsilon \,\right)\;.\]
Equivalently, for a given \,$\epsilon \,>\, 0$, there exist some \;$e \;\in\; X$\; and \,$\delta \;>\; 0$\, such that 
\[ x \;\in\; X \;,\; \left\|\, x \,-\, x_{\,0} \,,\, e \,\right\| \;<\; \delta \;\Rightarrow\; \left|\, T\,(\, x \,,\, b \,) \,-\, T\,(\, x_{\,0} \,,\, b \,) \,\right| \;<\; \epsilon\;.\]
\end{definition}

\begin{theorem}
Let \,$X$\, be a linear 2-normed space and \,$b \,\in\, X$\; be fixed\,.\;If a b-linear functional \,$T$\, on \,$X \,\times\, \left <\,b\,\right >$\; is continuous at \,$0$\, then it is continuous on the whole space \,$X$\,.
\end{theorem}

\begin{proof}
Let \,$T \,:\, X \,\times\, \left <\,b\,\right > \,\to\, \mathbb{K}$\; be a b-linear functional which is continuous at \,$0$\; and \,$x_{\,0} \;\in\; X$\; be arbitrary\,.\;Then for any open ball \,$ B\,(\, 0 \,,\, \epsilon \,)$\; in \;$\mathbb{K}$, we can find an open ball \;$ B_{\,e}\,(\, 0 \,,\, \delta \,)$\; in \;$X$\; such that 
\[ T\, \left(\, B_{\,e}\,(\, 0 \,,\, \delta \,) \,,\, b \,\right) \,\subseteq\, B\,(\, T\,(\, 0 \,,\, b\,) \,,\, \epsilon \,) \,=\, B\,(\, 0 \,,\, \epsilon \,)\; \;[\;\because\; T\,(\, 0 \,,\, b \,) \,=\, 0\;]\;.\] 
Then, 
\[ T\,(\, x \,,\, b \,) \,-\, T\,(\, x_{\,0} \,,\, b \,) \,=\, T\,(\,x \,-\, x_{\,0} \,,\, b\,) \,\in\, B\,(\, 0 \,,\, \epsilon \,)\;,\] whenever \,$ x \,-\, x_{\,0} \,\in\, B_{\,e}\,(\, 0 \,,\, \delta \,)$\,.\;Thus\,, if \,$x \,\in\, x_{\,0} \,+\, B_{\,e}\,(\, 0 \,,\, \delta \,) \,=\, B_{\,e}\,(\, x_{\,0} \,,\, \delta \,)$, then 
\[ T\,(\, x \,,\, b \,) \;\in\; T\,(\, x_{\,0} \,,\, b \,) \,+\, B\,(\, 0 \,,\, \epsilon \,) \,=\, B\, \left(\, T\,(\, x_{\,0} \,,\, b \,) \,,\, \epsilon \,\right)\;.\]Therefore, 
\[T\,\left(\, B_{\,e}\,(\, x_{\,0} \,,\, \delta \,) \,,\, b \,\right) \,\subseteq\, B\, \left(\, T\,(\, x_{\,0} \,,\, b \,) \,,\, \epsilon \,\right)\;.\]
Since \,$x_{\,0}$\, is arbitrary element of \,$X$, So \,$T$\, is continuous on \,$X$\,.
\end{proof}

\begin{theorem}
Let \,$X$\, be a linear 2-normed space\,.\;Then every continuous b-linear functional defined on \,$X \,\times\, \left <\,b\,\right >$\; is b-sequentially continuous\,.
\end{theorem}

\begin{proof}
Suppose that \,$T \,:\, X \,\times\, \left <\,b\,\right > \,\to\, \mathbb{K}$\; is continuous at \,$x \,\in\, X $\,.\;Then for any open ball \,$B\,\left(\, T\,(\, x \,,\, b \,) \,,\, \epsilon \,\right)$\; in \,$\mathbb{K}$, we can find an open ball $B_{\,e}\,(\, x \,,\, \delta \,)$\; in \,$X$\, such that 
\begin{equation}\label{eq3}
T\,\left(\, B_{\,e}\,(\, x \,,\, \delta \,) \,,\, b \,\right) \,\subseteq\, B\,\left(\, T\,(\, x \,,\, b \,) \,,\, \epsilon \,\right)\;. 
\end{equation} 
Let \,$\left\{\,x_{\,n}\,\right\}$\; be any sequence in \,$X$\, such that \,$x_{\,n} \,\to\, x $\; as \,$n \,\to\, \infty$\,.\;Then, for the open ball \,$B_{\,e}\,(\, x \,,\, \delta \,)$, there exist some \,$K \,>\, 0$\; such that \,$ x_{\,n} \,\in\, B_{\,e}\,(\, x \,,\, \delta \,)\; \;\forall\; n \,\geq\, K $\,.\;Now from (\ref{eq3}), it follows that
\[ T\,(\, x_{\,n} \,,\, b \,) \,\in\, B\,\left(\, T\,(\, x \,,\, b \,) \,,\, \epsilon \,\right)\; \;\forall\; n \,\geq\, K \]
\[\Rightarrow\; \left|\, T\,(\, x_{\,n} \,,\, b \,) \,-\, T\,(\, x \,,\, b \,)\,\right| \,<\, \epsilon \;\; \;\forall\; n \,\geq\, K\;.\] Since \,$B\,\left(\, T\,(\, x \,,\, b \,) \,,\, \epsilon \,\right)$\; is an arbitrary open ball in \,$\mathbb{K}$, it follows that \;$ T\,(\,x_{\,n} \,,\, b \,) \,\to\, T\,(\, x \,,\, b \,)$\; as \,$n \,\to\, \infty$\,.\;This shows that \,$T$\, is b-sequentially continuous on \,$X$\,.
\end{proof}

\begin{theorem}
Let \,$X$\, be a finite dimensional linear 2-normed space\,.\;Then every b-linear functional defined on \,$X \,\times\, \left <\,b\,\right >$\; is b-sequentially continuous\,.
\end{theorem}

\begin{proof}
Let \,$X$\, be a finite dimensional linear 2-normed space and \,$T \,:\, X \,\times\, \left <\,b\,\right > \,\to\, \mathbb{K}$\; be a b-linear functional\,.\;If \,$ X \,=\, \{\,0\,\}$\; then the proof is obvious\,.\;Suppose that \,$ X \,\neq\, \{\,0\,\}$\; and let \,$\left\{\; e_{\,1} \;,\; e_{\,2} \;,\; \;\cdots\; \; e_{\,m}\,\right\}$\; be a basis for \,$X$\,.\;We now show that \,$T$\; is b-sequentially continuous\,.\;Let \,$\{\,x_{\,n}\,\}$\; be a sequence in \,$X$\, with \,$ x_{\,n} \,\to\, x $\; as \,$n \,\to\, \infty$\,.\;Now, we can write\,, 
\[ x_{\,n} \;=\; a_{\,n \,,\, 1}\,e_{\,1} \;+\; a_{\,n \,,\, 2}\,e_{\,2} \;+\; \;\cdots \;+\; a_{\,n \,,\, m}\,e_{\,m}\]and 
\[x \;=\; a_{\,1}\,e_{\,1} \;+\; a_{\,2}\,e_{\,2} \;+\; \;\cdots \;+\; a_{\,m}\,e_{\,m}\hspace{.8cm}\]  
where \,$ a_{\,n \,,\, j} \;,\; a_{\,1} \;,\; a_{\,2} \;\cdots \;a_{\,m} \,\in\, \mathbb{R}$\,.\;In Theorem (\ref{th2}), it has been shown that \,$ a_{\,n \,,\, j} \,\to\, a_{\,j}$\; as \,$n \,\to\, \infty$\; for all \,$\,j$\,.\;Then
\[ T\,(\, x_{\,n} \,,\, b \,) \,=\, T\,\left(\, a_{\,n \,,\, 1}\,e_{\,1} \,+\, a_{\,n \,,\, 2}\,e_{\,2} \,+\, \cdots \,+\, a_{\,n \,,\, m}\,e_{\,m} \;,\; b \,\right) \hspace{2cm}\]
\[ {\hspace{2cm}} \,=\, a_{\,n\,,\,1}\,T\,(\,e_{\,1} \,,\, b \,) \,+\, a_{\,n\,,\,2}\,T\,(\, e_{\,2} \,,\, b \,) \,+\, \cdots \,+\, a_{\,n\,,\,m}\,T\,(\, e_{\,m} \,,\, b \,) \]
\[ {\hspace{2.7cm}} \,\to\, \,a_{\,1}\,T\,(\, e_{\,1} \,,\, b \,) \,+\, a_{\,2}\,T\,(\, e_{\,2} \,,\, b \,) \,+\,  \cdots \,+\, a_{\,m}\,T\,(\, e_{\,m} \,,\, b \,) \,,\, \text{as}\; n \,\to\, \infty \]
\[ {\hspace{.7cm}} \,=\, T\,\left(\, a_{\,1}\,e_{\,1} \,+\, a_{\,2}\,e_{\,2} \,+\, \cdots \,+\, a_{\,m}\,e_{\,m} \,,\, b \,\right) \,=\, T\,(\, x \,,\, b \,)\;.\]
Thus, we have shown that if \,$ x_{\,n} \,\to\, x \,\Rightarrow\, T\,(\,x_{\,n} \,,\, b \,) \,\to\, T\,(\, x \,,\, b \,)$\,.\;Therefore, \,$T$\, is b-sequentially continuous\,.
\end{proof}

%=====================================
\section{Results in 2-Banach space analogous to Uniform Boundedness Principle and Hahn-Banach Theorem}
%=====================================

In this section we give the notion of Pointwise boundedness and Uniformly boundedness of a bounded b-linear functional in linear 2-normed space\,.\;We derive an analogue of Uniform Boundedness Principle and Hahn-Banach Extension Theorem for bounded b-linear functional on 2-normed space\,.\;We define weak\,*\,convergence of sequence of bounded b-linear functionals in linear 2-normed spaces\,.

\begin{definition}
A subset \,$A$\, of a linear 2-normed space \,$X$\, is said to be nowhere dense if its closure has empty interior\,.\;Thus, \,$A$\, is nowhere dense in \,$X$\, if corresponding to any open ball \,$B_{\,e}\,(\, a \,,\, \delta \,)$\; in \,$X$\, with \,$a \,\in\, A $, there exists another open ball \,$B_{\,e}\,(\, a^{\,'} \,,\, \delta^{\,'} \,)$\; such that \,$B_{\,e}\,(\, a^{\,'} \,,\, \delta^{\,'} \,) \,\subset\, B_{\,e}\,(\, a \,,\, \delta \,)$\; with \,$ A \,\cap\, B_{\,e}\,(\, a^{\,'} \,,\, \delta^{\,'} \,) \,=\, \phi$\,.
\end{definition}

\begin{definition}
Let \,$X$\; be a linear 2-normed space\,.\;A set \,$\mathcal{A}$\; of bounded b-linear functionals defined on \,$X \,\times\, \left <\,b\,\right >$\; is said to be:

\begin{itemize}
\item[(a)]\;\; Pointwise bounded if for each \,$x \,\in\, X $, the set \,$\left\{\,T\,(\, x \,,\, b \,) \,:\, T \,\in\, \;\mathcal{A} \,\right\}$\; is a bounded set in \;$\mathbb{K}$\,.\;That is 
\[\left|\, T\,(\, x \,,\, b \,)\,\right| \;\leq\; K \,\left\|\, x \,,\, b \,\right\|\;\; \;\forall\; x \;\in\; X \;\;\&\;\; \forall\; T \;\in\; \mathcal{A}\;.\]
\item[(b)]\;\; Uniformly bounded if there is a constant \,$K \,>\, 0$ such that \,$\|\,T\,\| \,\leq\, K\;\; \;\forall\; T \,\in\, \mathcal{A}$\,.
\end{itemize}

\end{definition}

\begin{theorem}
If a set \,$\mathcal{A}$\; of bounded b-linear functionals on \,$X \,\times\, \left <\,b\,\right >$\; is uniformly bounded then it is pointwise bounded set\,.
\end{theorem}

\begin{proof}
Suppose \,$\mathcal{A}$\; uniformly bounded\,.\;Then there is a constant \,$K \,>\, 0$\, such that \,$\|\,T\,\| \;\leq\; K \;\; \;\forall\; T \;\in\; \mathcal{A}$\,.\;Let \,$x \;\in\; X$\, be given\,. Then, 
\[\left|\,T\,(\, x \,,\, b \,)\,\right| \,\leq\, \|\,T \,\| \,\left\| \,x \,,\, b \,\right\| \,\leq\, K\, \,\left\|\, x \,,\, b \,\right\| \; \;\forall\; T \;\in\; \mathcal{A}\]
and hence \,$\mathcal{A}$\, is poinwise bounded set in \,$\mathbb{K}$\,.
\end{proof}

\begin{theorem}\label{th6}
Let \,$X$\; be 2-Banach space over the field \,$\mathbb{K}$\; and \,$ b \,\in\, X $\; be fixed\,.\;If a set \,$\mathcal{A}$\; of bounded b-linear functionals on \,$ X \,\times\, \left <\,b\,\right > $\; is pointwise bounded, then it is uniformly bounded\,.
\end{theorem}

\begin{proof}
For each positive integer \,$n$, we consider the set 
\[ F_{n} \,=\, \left \{\, x \,\in\, X \,:\, \left|\,T\,(\, x \,,\, b \,) \,\right| \,\leq\, n \; \;\forall\; T \;\in\; \mathcal{A} \;\right \} \;.\]
We now show that \,$F_{n}$\; is a closed subset of \,$X$\,.\;Let \,$x \,\in\, \overline{F_{n}}$\; and \,$\{\, x_{\,k} \,\}$\; be a sequence in \,$F_{n}$\, such that \, $ x_{\,k} \,\to\, x $\; as \;$ k \,\to\, \infty $\,.\;Then 
\[\left|\,T\,(\, x_{\,k} \,,\, b\,)\,\right| \;\leq\; n \; \;\forall\; T \;\in\; \mathcal{A}\;.\]
Now, by the Theorem (\ref{th5}), \,$T$\, is b-sequentially continuous\,.\;So 
\[ \lim\limits_{k \;\to\; \infty} \left|\, T\,(\,x_{\,k} \,,\, b\,) \,\right| \,=\, T\,(\,x \,,\, b\,)\;.\]
This shows that \,$|\,T\,(\, x \,,\, b\,)\,| \,\leq\, n \; \;\;\forall\; T \,\in\, \mathcal{A} \,\Rightarrow\, x \,\in\, F_{\,n}$\; and hence \,$F_{\,n}$\; becomes a closed subset of \,$X$\, for every \,$n \,\in\, \mathbb{N}$\,.\;Since \,$\mathcal{A}$\; is pointwise bounded, then the set \,$\left\{\,T\,(\,x \,,\, b\,) \,:\, T \,\in\, \mathcal{A}\,\right\}$\; is a bounded for each \,$ x \,\in\, X $\,.\;Thus, we see that for each \,$x \,\in\, X$\; is in some \;$F_{n}$\; and therefore
\[ X \;=\; \;\bigcup\limits_{n \,=\, 1}^{\infty}\; F_{n}\;.\]
Since \,$X$\, is 2-Banach space, by Baire's Category theorem for 2-Banach space, \,$\exists\; \,n_{\,0} \,\in\, \mathbb{N}$\; such that \,$F_{n_{\,0}}$\; is not nowhere dense in \,$X$\; i\,.\,e \,$F_{\,n_{\,0}}$\; has nonempty interior\,. Consequently, \,$\exists$\; a non-empty open ball \,$B_{\,e}\,(\,x_{\,0} \,,\, \delta\,)$\; such that \,$B_{\,e}\,(\,x_{\,0} \,,\, \delta\,) \,\subset\, F_{n_{\,0}}$\,. i\,.\,e\,, 
\[|\,T\,(\,x \,,\, b \,)\,| \;\leq\; n_{\,0}\;\;\;\forall\; x \;\in\; B_{\,e}\,(\,x_{\,0} \,,\, \delta\,) \; \;\text{and}\; \; \forall\; T \;\in\; \mathcal{A}\;.\]
The above expression can be written in the form 
\[\left|\,T\,\left(\,B_{\,e}\,(\,x_{\,0} \,,\, \delta\,) \,,\, b\,\right)\,\right| \,\leq\, n_{\,0} \; \;\forall\; T \,\in\, \mathcal{A}\;.\]Note that 
\[x_{\,0} \,+\, \delta\, B_{\,e}\,(\,0 \,,\, 1\,) \,=\, \left\{\,x \,\in\, X \,:\, x \,=\, x_{\,0} \,+\, \delta\, a \;\;,\;\; a \,\in\, B_{\,e}\,(\,0 \,,\, 1\,)\,\right\}\]
\[{\hspace{2.5cm}} \;=\; \left\{\,x \,\in\, X \,:\, x \,=\, x_{\,0} \,+\, \delta\, a \;\;,\; \left\|\,a \,,\, e\,\right\| \,<\, 1 \,\right\}\]
\[\hspace{1cm}\;=\, \left\{\,x \,\in\, X \,:\, \left\|\,\dfrac{x \,-\, x_{\,0}}{\delta} \,,\, e \,\right\| \,<\, 1 \,\right\}\]
\[\hspace{3.4cm} \,=\, \left\{\, x \,\in\, X \,:\, \left\|\,x \,-\, x_{\,0} \,,\, e\,\right\| \,<\, \delta\,\right\} \,=\, B_{\,e}\,(\,x_{\,0} \,,\, \delta\,)\]
\[\Rightarrow\, B_{\,e}\,(\,0 \,,\, 1\,) \,=\, \dfrac{B_{\,e}\,(\,x_{\,0} \,,\,\delta \,) \,-\, x_{\,0}}{\delta}\;.\hspace{3.5cm}\]
Clearly, 
\[ |\,T\,(\,x_{\,0} \,,\, b \,)\,| \,\leq\, n_{\,0} \; \;\;\forall\; T \,\in\, \mathcal{A}\; \;[\;\because\; x_{\,0} \,\in\, B_{\,e}\,(\,x_{\,0} \,,\, \delta\,)\;]\;.\]Therefore,
\[\left|\,T\,(\,B_{\,e}\,(\,0 \,,\, 1\,) \,,\, b\,)\,\right| \,=\, \left|\,T\,\left(\,\dfrac{B_{\,e}\,(\,x_{\,0} \,,\, \delta\,) \,-\, x_{\,0}}{\delta} \,,\, b\,\right)\,\right|\hspace{1cm}\]
\[{\hspace{2.5cm}} \;=\; \left|\,\dfrac{1}{\delta}\,T\, \left(\,B_{\,e}\,(\,x_{\,0} \,,\, \delta\,) \,-\, x_{\,0} \,,\, b\,\right)\,\right|\]
\[ \hspace{4.8cm} \,\leq\, \dfrac{1}{\delta}\, \left\{\,\left|\,T\,(\, B_{\,e}\,(\,x_{\,0} \,,\, \delta\,) \,,\, b\,)\,\right| \,+\, |\,T\,(\, x_{\,0} \,,\, b\,)\,|\,\right\}\] 
\[ \hspace{.7cm} \;\leq\; \dfrac{2 \,n_{\,0}}{\delta}\; \;\;\forall\; T \;\in\; \mathcal{A}\] 
\[\Rightarrow\, |\,T\,(\,x \,,\, b \,) \;| \;\leq\; \dfrac{2\,n_{\,0}}{\delta} \;\; \;\forall\; x \;\in\; B_{\,e}\,(\, 0 \,,\, 1\,) \; \;\text{and}\; \;\forall\; T \;\in\; \mathcal{A}\;.\]Thus,
\[\|\,T\,\| \,=\, \sup\left\{\,\left|\,T\,(\,x \,,\, b\,)\,\right| \,:\, x \,\in\, B_{\,e}\,(\, 0 \,,\, 1\,)\;\; \;\;\forall\; T \,\in\, \mathcal{A} \,\right\} \,\leq\, \dfrac{2\,n_{\,0}}{\delta}\; \;\forall\; T \,\in\, \mathcal{A}\;.\]
This proves that \,$\mathcal{A}$\; is uniformly bounded\,.
\end{proof}

\begin{theorem}
Let \,$X$\; be a 2-Banach space and \,$X^{\,\ast}_{\,b}$\; be the Banach space of all bounded b-linear functionals defined on \,$X \,\times\, \left <\,b\,\right >$\,.\;If \,$\left\{\,T_{\,n}\,\right\} \,\subseteq\, X^{\,\ast}_{\,b}$\; be a sequence such that
\begin{equation}\label{eqn3} 
\lim\limits_{n \;\to\; \infty}\,T_{\,n}\,(\,x \,,\, b\,) \,=\, T\,(\,x \,,\, b\,)\;  \;\;\forall\; x \;\in\; X
\end{equation}
exists, then \;$T \;\in\; X^{\,\ast}_{\,b}$\,.
\end{theorem}

\begin{proof}
Note that (\ref{eqn3}) defines a mapping \,$T \,:\, X \,\times\, \left<\,b\,\right> \,\to\, \mathbb{K}$\; which is, clearly, a b-linear functional\,.\;We need only to show that \,$T$\; is bounded\,.\\Since, for every \,$x \,\in\, X$, \,$\left\{\,T_{\,n}\,(\,x \,,\, b\,)\,\right\}$\; convergent sequence in \,$\mathbb{K}$, then it is bounded in \,$\mathbb{K}$\,.\;By the Theorem (\ref{th6}), the set \,$\left\{\,\left\|\,T_{\,n}\,\right\|\,\right\}$\; is bounded\,.\;Then there exists some constant \,$M \,>\, 0$\; such that \,$\left\|\,T_{\,n}\,\right\| \;\leq\; M \;\; \;\forall\; n \;\in\; \mathbb{N}$\,.\;Therefore, 
\[\left|\,T_{\,n}\,(\,x \,,\, b\,)\,\right| \,\leq\, \,\left\|\,T_{\,n}\,\right\|\,  \left\|\,x \,,\, b\,\right\| \,\leq\, M\, \left\|\,x \,,\, b\,\right\|\;\; \;\forall\; x \,\in\, X \;\;\&\;\; \;\forall\; n \,\in\, \mathbb{N}\]
\[ \Rightarrow\; \lim\limits_{n \;\to\; \infty}\, \left|\,T_{\,n}\,(\,x \,,\, b\,)\,\right| \;\leq\;  M\, \left\|\,x \,,\, b\,\right\| \;\; \;\forall\; x \;\in\; X \]
\[ \Rightarrow\; |\,T\,(\, x \,,\, b\,)\,| \;\leq\;  M\, \left\|\,x \,,\, b\,\right\|\;\; \;\forall\; x \,\in\, X\;. \hspace{1.1cm}\]
This shows that \,$T$\; is bounded b-linear functional defined on \,$X \,\times\, \left <\,b\,\right >$\; and hence \,$T \,\in\, X^{\,\ast}_{\,b}$\,. 
\end{proof}

\begin{definition}
A sequence \,$\left\{\, T_{\,n} \,\right\}$\; in \,$X^{\,\ast}_{\,b}$\,, where \,$X^{\,\ast}_{\,b}$\;is the Banach space of all bounded b-linear functionals defined on \,$X \,\times\, \left <\,b\,\right >$\; is said to be b-weak\,*\,convergent if there exists an \,$T \,\in\, X^{\,\ast}_{\,b}$\; such that 
\[\lim\limits_{n \;\to\; \infty}\,T_{\,n}\,(\, x \,,\, b \,) \,=\, T\,(\,x \,,\, b \,) \;\;\forall\; x \;\in\; X\;.\] 
The limits \,$T$\, is called the b-weak\,*\,limit of the sequence \,$\left\{\, T_{\,n}\,\right\}$\,. 
\end{definition}

\begin{definition}
A subset \,$M$\, of a linear 2-normed space \,$X$\, is said to be a fundamental or total set if the set \,$\textit{Span}\,M$\; is dense in \,$X$, that is \,$\overline{\textit{Span}\,M} \,=\, X$\,.
\end{definition}

\begin{theorem}
Let \,$X$\, be a 2-Banach space and \,$\left\{\,T_{\,n}\,\right\} \,\subseteq\, X^{\,\ast}_{\,b}$\; be a sequence\,.\;Then \,$\left\{\,T_{\,n}\,\right\}$\; is b-weak\,*\,Convergent if and only if the following two conditions hold:
\begin{itemize}
\item[(a)]\;\; The sequence \,$\left\{\,\|\,T_{\,n}\,\|\,\right\}$\; is bounded and
\item[(b)]\;\; The sequence \,$\left\{\,T_{\,n}\,(\,x \,,\, b\,)\,\right\}$\; is Cauchy sequence for each \,$ x \,\in\, W$, where \,$W$\, is fundamental or total subset of \,$X\,$\,.
\end{itemize}
\end{theorem}

\begin{proof}
Let \,$\left\{\,T_{\,n}\,\right\}$\; be b-weak\,*\,Convergent in \,$X^{\,\ast}_{\,b}$\,.\;Then
\[\lim\limits_{n \;\to\; \infty}\,T_{\,n}\,(\, x \,,\, b\,) \,=\, T\,(\,x \,,\, b \,)\; \;\forall\; x \;\in\; X\;.\]
This implies that \,$\left\{\,T_{\,n}\,(\,x \,,\, b\,)\,\right\}$\; is bounded for each \,$ x \,\in\, X $\,.\;Since \,$X$\, is 2-Banach space then by Theorem (\ref{th6}), we get that \,$\left\{\,\|\, T_{\,n}\,\|\,\right\}$\; is bounded and therefore \,$(\,a\,)$\; hold\,.\;Now from the definition of b-weak\,*\,Convergence, \,$\left\{\, T_{\,n}\,(\,x \,,\, b \,)\,\right\}$\; is a convergent sequence of numbers for \,$x \,\in\, X$\,, in particular, for \,$ x \,\in\, W $\,.\;This proves \;$(\,b\,)$\,.\\

\textit{Conversely}, suppose that the given two conditions hold\,.\;Since the sequence \,$\left\{\,\|\,T_{\,n}\,\|\,\right\}$\; is bounded, \,$\exists$\; a constant \,$K \,>\, 0$\; such that \,$\|\, T_{\,n}\,\| \,\leq\, K \; \;\forall\; n \,\in\, \mathbb{N}$\,.\;Also, \,$\overline{\textit{\,Span\,W}} \,=\, X$, then for a given \,$\epsilon \,>\, 0$\; and for each \,$x \,\in\, X \;,\; \;\exists\; \;y \,\in\, \;\textit{Span\,W}$\; such that \;$\left\|\,x \,-\, y \,,\, b\,\right\| \;<\; \dfrac{\epsilon}{3\,K}$\,.\;Now from the condition \,$(\,b\,)$, it follows that \,$\left\{\,T_{\,n}\,(\,y \,,\, b\,)\,\right\}$\; is Cauchy sequence for \,$y \,\in\, \textit{Span\,W} $\; and hence there exists an integer \,$N \;>\; 0$\; such that
\[\left|\,T_{\,n}\,(\,y \,,\, b\,) \,-\, T_{\,m}\,(\,y \,,\, b\,)\right| \,<\, \dfrac{\epsilon}{3}\;\; \;\forall\; m \;,\; n \;\geq\; N\;.\]
Now, for an arbitrary element \,$x \,\in\, X$, we have 
\[\left|\,T_{\,n}\,(\,x \,,\, b\,) \,-\, T_{\,m}\,(\,x \,,\, b\,)\,\right|\]
\[ \,=\, \left|\,T_{\,n}\,(\,x \,,\, b\,) \,-\, T_{\,n}\,(\,y \,,\, b\,) \,+\, T_{\,n}\,(\,y \,,\, b\,) \,-\, T_{\,m}\,(\,y \,,\, b\,) \,+\, T_{\,m}\,(\,y \,,\, b\,) \,-\, T_{\,m}\,(\,x \,,\, b\,)\,\right| \] 
\[ \,\leq\, \left|\,T_{\,n}\,(\,x \,,\, b\,) \,-\, T_{\,n}\,(\,y \,,\, b\,)\,\right| \,+\, \left|\,T_{\,n}\,(\,y \,,\, b\,) \,-\, T_{\,m}\,(\,y \,,\, b\,)\,\right| \,+\, \left|\,T_{\,m}\,(\,y \,,\, b\,) \,-\, T_{\,m}\,(\,x \,,\, b\,)\,\right|\]
\[\,<\, \|\,T_{\,n}\,\|\, \left\|\,x \,-\, y \,,\, b\,\right\| \,+\, \left|\,T_{\,n}\,(\,y \,,\, b\,) \,-\, T_{\,m}\,(\,y \,,\, b\,)\,\right| \,+\, \|\,T_{\,m}\,\|\, \left\|\,x \,-\, y \,,\, b\,\right\| \hspace{2cm}\] 
\[\;<\; K\, \cdot\, \dfrac{\epsilon}{3\,K} \,+\, \dfrac{\epsilon}{3} \,+\, K\, \cdot\, \dfrac{\epsilon}{3\,K} \,=\, \epsilon \; \;\;\forall\; m \;,\; n \;\geq\; N \;.\hspace{5.3cm}\]
Therefore,
\[\left|\,T_{\,n}\,(\,x \,,\, b\,) \,-\, T_{\,m}\,(\,x \,,\, b\,)\,\right| \;<\; \epsilon \; \;\;\forall\; m \;,\; n \;\geq\; N\;. \]
This shows that \,$\left\{\,T_{\,n}\,(\,x \,,\, b\,)\,\right\}$\; is Cauchy sequence in \,$\mathbb{K}$\,.\;But \,$\mathbb{K}$\; is complete, So \,$\left\{\,T_{\,n}\,(\,x \,,\, b\,)\,\right\}$\; is converges to \,$T\,(\,x \,,\, b\,)$\; in \,$\mathbb{K}$\,.\;Since \,$x$\, is an arbitrary element of \,$X$, therefore it follows that  
\[\lim\limits_{n \to \infty}\,T_{\,n}\,(\,x \,,\,b\,) \,=\, T\,(\,x \,,\, b\,) \; \;\forall\; x \,\in\, X\;.\] 
Thus, \,$\left\{\,T_{\,n}\,\right\}$\; is b-weak\,*\,Convergent sequence in \,$X^{\,\ast}_{\,b}$\,.
\end{proof}

\begin{definition}
A linear 2-normed space \,$X$\, is said to be separable if \,$X$\, has a countable dense subset\,.
\end{definition}

\begin{theorem}\label{th7}
Let \,$X$\; be a linear 2-normed space over the field \,$\mathbb{R}$\; and \,$W$\, be a subspace of $\,X$\,.\;Then each bounded b-linear functional \,$T_{\,W}$\, defined on \,$W \,\times\, \left <\,b\,\right >$\; can be extended onto \,$X \,\times\, \left <\,b\,\right >$\; with preservation of the norm\,.\;In other words, there exists a bounded b-linear functional \,$T$\, defined on \,$X \,\times\, \left <\,b\,\right >$\; such that
\[T\,(\,x \,,\, b\,) \,=\, T_{\,W}\,(\,x \,,\, b\,) \;\;\forall\; x \,\in\, W\;  \;\;\&\;\; \left\|\,T_{\,W}\,\right\| \,=\, \|\,T\,\|\;.\] 
\end{theorem}

\begin{proof}
We prove this theorem by assuming \,$X$\, is separable\,.\;This theorem also hold for the spaces which are not separable\,.\;Let \;$x_{\,0} \,\in\, X \,-\, W$\; and consider the set \,$W \,+\, x_{\,0}  \,=\, \{\,x \,+\, t\,x_{\,0} \,:\, x \,\in\, W$\; and \,$t$\, is arbitrary real number $\}$\,.\;Clearly, \,$ W \,+\, x_{\,0}$\; is a subspace of \,$X$\, containing \,$W$\,.\;Let \,$T_{\,W}$\, be a bounded b-linear functional defined on \,$W \,\times\, \left <\,b\,\right >$\, and further suppose that \,$(\,x_{\,1} \,,\, b\,) \;,\; (\,x_{\,2} \,,\, b\,) \,\in\, W \,\times\, \left <\,b\,\right >$\,.\;Now, 
\[ T_{\,W}\,(\,x_{\,1} \,,\, b\,) \,-\, T_{\,W}\,(\,x_{\,2} \,,\, b\,) \,\leq\, \left\|\,T_{\,W}\,\right\|\, \left\|\,x_{\,1} \,-\, x_{\,2} \,,\, b\,\right\| \]
\[\,\leq\, \left\|\,T_{\,W}\,\right\|\, \left(\,\left\|\,x_{\,1} \,+\, x_{\,0} \,,\, b\,\right\| \,+\, \left\|\,x_{\,2} \,+\, x_{\,0} \,,\, b\,\right\|\,\right) \] 
\[\Rightarrow\, \,T_{\,W}\,(\,x_{\,1} \,,\, b\,) \,-\, \left\|\,T_{\,W}\,\right\| \, \left\|\,x_{\,1} \,+\, x_{\,0} \,,\, b\,\right\| \,\leq\, T_{\,W}\,(\,x_{\,2} \,,\, b\,) \,+\, \left\|\,T_{\,W}\,\right\|\, \left\|\,x_{\,2} \,+\, x_{\,0} \,,\, b\,\right\|\;.\]
Since \,$(\,x_{\,1} \,,\, b\,) \,,\, (\,x_{\,2} \,,\, b\,)$\; be two arbitrary elements in \,$W \,\times\, \left<\,b\,\right>$, we obtain 
\[ \sup\limits_{x \;\in\; W} \left\{\, T_{\,W}\,(\,x \,,\, b\,) \,-\, \left\|\,T_{\,W}\,\right\|\, \left\|\,x \,+\, x_{\,0} \,,\, b\,\right\|\,\right\} \]
\[\,\leq\, \inf\limits_{x \;\in\; W} \left\{\, T_{\,W}\,(\,x \,,\, b\,) \,+\, \left\|\,T_{\,W}\,\right\|\, \left\|\,x \,+\, x_{\,0} \,,\, b\,\right\|\,\right\} \]
and hence we can find a real number \,$\alpha$\, such that
\[\sup\limits_{x \;\in\; W} \left\{\,T_{\,W}\,(\,x \,,\, b\,) \,-\, \left\|\,T_{\,W}\,\right\|\, \left\|\,x \,+\, x_{\,0} \,,\, b\,\right\|\,\right\} \,\leq\, \alpha \hspace{6cm}\]
\begin{equation} \label{eq4}
\hspace{3cm}\,\leq\, \inf\limits_{x \;\in\; W} \left\{\,T_{\,W}\,(\,x \,,\, b\,) \,+\, \left\|\,T_{\,W}\,\right\|\, \left\|\,x \,+\, x_{\,0} \,,\, b\,\right\|\,\right\}\;. 
\end{equation}
Now we define a b-linear functional \,$T_{\,0}$\, on \,$\left(\,W \,+\, x_{\,0}\,\right) \,\times\, \left <\,b\,\right >$\; by, 
\[ T_{\,0}\,(\,y \,,\, b\,) \,=\, T_{\,W}\,(\,x \,,\, b\,) \,-\, t\,\alpha \;\; \;\forall\; (\,y \,,\, b\,) \,\in\, \left(\,W \,+\, x_{\,0}\,\right) \,\times\, \left <\,b\,\right >\]
where \,$y \,=\, x \,+\, t\,x_{\,0} \;,\; t$\, is a unique real number and \,$\alpha$\, is the real number satisfying (\ref{eq4}) and \,$x \,\in\, W$\,.\;Clearly, \,$T_{\,W}\,(\,y \,,\, b\,) \,=\, T_{\,0}\,(\,y \,,\, b\,) \; \;\forall\; y \,\in\, W $\,.\;We now show that \,$T_{\,0}$\, is bounded on \,$\left(\, W \,+\, x_{\,0}\,\right) \,\times\, \left <\,b\,\right >$\; and \,$\left\|\,T_{\,W}\,\right\| \,=\, \left\|\,T_{\,0}\,\right\|$\,.\;For the boundedness part of \,$T_{\,0}$\,, we consider the following two cases
\begin{itemize}
\item[(i)]\;\;  First we consider \,$t \,>\, 0$\,.\;Since \,$W$\; is a subspace, we get \,$\dfrac{x}{t} \,\in\, W$, whenever \,$x \,\in\, W$\; and then the inequality (\,\ref{eq4}\,) implies that,
\[ T_{\,0}\,(\,y \,,\, b\,) \,=\, t \cdot\, \left\{\,\dfrac{1}{t}\,T_{\,W}\,(\,x \,,\, b\,) \,-\, \alpha \,\right\} \,=\, t \cdot\, \left\{\,T_{\,W}\,\left(\,\dfrac{x}{t} \,,\, b\,\right) \,-\, \alpha \,\right\} \]
\[ \,\leq\, t \cdot\, \left\|\,T_{\,W}\,\right\|\, \left\|\, \dfrac{x}{t} \,+\, x_{\,0} \,,\, b\,\right\|  \hspace{3cm} \]
\[\,=\, \left\|\,T_{\,W}\,\right\|\, \left\|\,x \,+\, t\,x_{\,0} \,,\, b\,\right\| \,=\, \left\|\,T_{\,W}\,\right\|\, \left\|\,y \,,\, b\,\right\|\;.\]
\item[(ii)]\;\; Next we consider \,$t \,<\, 0$\, and again using the inequality (\ref{eq4}\,)  
\[ T_{\,W}\,\left(\,\dfrac{x}{t} \,,\, b\,\right) \,-\, \alpha\,\,\geq\, \,-\, \left\|\,T_{\,W}\,\right\|\, \left\|\,\dfrac{x}{t} \,+\, x_{\,0} \,,\, b\,\right\| \hspace{2.5cm}\]
\[ \hspace{3.8cm} \,=\, \,-\,\dfrac{1}{|\,t\,|}\, \left\|\,T_{\,W}\,\right\|\, \left\|\,y \,,\, b\,\right\| \,=\, \dfrac{1}{t}\, \left\|\,T_{\,W}\,\right\|\, \left\|\,y \,,\, b\,\right\| \;.\]
So, 
\[T_{\,0}\,(\,y \,,\, b\,) \,=\, t \cdot\, \left\{\,T_{\,W}\,\left(\,\dfrac{x}{t} \,,\, b\,\right) \,-\, \alpha\,\right\} \,\leq\, t \cdot\, \dfrac{1}{t}\, \left\|\,T_{\,W}\,\right\|\, \left\|\,y \,,\, b\,\right\| \,=\, \left\|\,T_{\,W}\,\right\|\, \left\|\,y \,,\, b\,\right\|\;.\]
\end{itemize}
Therefore,
\begin{equation} \label{eq5}
T_{\,0}\,(\,y \,,\, b\,) \,\leq\, \left\|\,T_{\,W}\,\right\|\, \|\,y \,,\, b\,\| \; \;\;\forall\; (\,y \,,\, b\,) \,\in\, \left(\, W \,+\, x_{\,0} \,\right) \,\times\, \left <\,b\,\right >\;.
\end{equation}
Replacing \,$ \,-\, y$\, for \,$y$\, in (\,\ref{eq5}\,), we get 
\[ T_{\,0}\,(\,-\, y \,,\, b\,) \,\leq\, \left\|\,T_{\,W}\,\right\| \, \|\,-\, y \,,\, b\,\| \,\Rightarrow\, \,-\, T_{\,0}\,(\,y \,,\, b\,) \,\leq\, \left\|\,T_{\,W}\,\right\|\, \|\,y \,,\, b\,\|\;.\]
Combining this with (\,\ref{eq5}\,), we obtain
\[ \left|\,T_{\,0}\,(\,y \,,\, b\,)\,\right| \,\leq\, \left\|\,T_{\,W}\,\right\|\, \|\,y \,,\, b\,\| \; \;\;\forall\; (\,y \,,\, b\,) \,\in\, \left(\, W \,+\, x_{\,0} \,\right) \,\times\, \left <\,b\,\right >\;.\]
This shows that \,$T_{\,0}$\; is bounded and \,$\left\|\,T_{\,0}\,\right\| \,\leq\, \left\|\,T_{\,W}\,\right\|$\,.\;Since the domain of \,$T_{\,W}$\, is a subset of the domain of \,$T_{\,0}$\,, we get \,$\left\|\,T_{\,0}\,\right\| \,\geq\, \left\|\,T_{\,W}\,\right\|$\, and hence \,$\left\|\,T_{\,0}\,\right\| \,=\, \left\|\,T_{\,W}\,\right\|$\,.\;Thus we have seen that \,$T_{\,0}\,(\,x \,,\, b\,)$\; is the extension of \,$T_{\,W}\,(\,x \,,\, b\,)$\; onto \;$\left(\, W \,+\, x_{\,0} \,\right) \,\times\, \left <\,b\,\right >$\; with \;$\left\|\,T_{\,0}\,\right\| \,=\, \left\|\,T_{\,W}\,\right\|$\,.\;Since \,$X$\, is separable, so there exists a countable dense subset \,$D$\, of \,$X$\,.\;We select elements from this dense subset those belong to \,$X \,-\, W$\, and arrange them as a sequence \,$\left\{\,x_{\,0} \;,\; x_{\,1} \;,\; x_{\,2} \,\cdots\, \right\}$\,.\;By the previous procedure, we get the extension of \,$T_{\,W}\,(\,x \,,\, b\,)$\; onto \,$\left(\,W \,+\, x_{\,0} \,\right) \,\times\, \left <\,b\,\right > \,=\, W_{\,1} \,\times\, \left <\,b\,\right > \,,\, \left(\, W_{\,1} \,+\, x_{\,1} \,\right) \,\times\, \left <\,b\,\right > \,=\, W_{\,2} \,\times\, \left <\,b\,\right > \,,\, \left(\,W_{\,2} \,+\, x_{\,2}\,\right) \,\times\, \left <\,b\,\right > \,=\, W_{\,3} \,\times\, \left <\,b\,\right >$\; and so on\,.\;Then we arrive at a bounded b-linear functional \,$T_{g} \,:\, W_{g} \,\times\, \left <\,b\,\right > \,\to\, \mathbb{K}$\,, where \,$W_{g}$\; is everywhere dense in \,$X$\, and that contains \;$W_{n}$\; for \;$n \,=\,1 \,,\, 2 \,,\, 3 \,\cdots$\, and \;$\left\|\,T_{g}\,\right\| \,=\, \left\|\,T_{\,W}\,\right\|$\,.\;If \,$y \,\in\, X \,-\, W_{g}$\,, then there exists a sequence \,$\left\{\,y_{\,n}\,\right\} $\; in \;$W_{g}$\; such that \;$y \,=\, \lim\limits_{n \to \infty} \,y_{\,n}$\,.\;We now define 
\[ T\,(\,y \,,\, b\,) \,=\, \lim\limits_{n \to \infty}\,T_{g}\,(\,y_{\,n} \,,\, b\,)\;.\]
If \,$y \,\in\, W_{g}$\,, we can put in particular \,$y_{\,1} \,=\, y_{\,2} \,=\, \,\cdots\, \,=\, y$\; and so the b-linear functional \;$T\,(\,y \,,\, b\,)$\; is an extension of \,$T_{g}\,(\,y \,,\, b\,)$\; onto \,$X \,\times\, \left <\,b\,\right >$\,.\;Now
\[ \left|\,T\,(\,y \,,\, b\,)\,\right| \,=\, \lim\limits_{n \to \infty} \left|\,T_{\,g}\,(\,y_{\,n} \,,\, b\,)\,\right| \,\leq\, \left\|\,T_{\,g}\,\right\|\, \lim\limits_{n \to \infty}\, \left\|\,y_{\,n} \,,\, b\,\right\| \,=\, \left\|\,T_{\,W}\,\right\|\, \left\|\,y \,,\, b\,\right\|\;.\]
This shows that \,$T\,(\,y \,,\, b\,)$\; is bounded b-linear functional and \,$\|\,T\,\| \,\leq\, \left\|\,T_{\,W}\,\right\|$\,.\;Since the domain of \,$T_{\,W}$\, is a subset of the domain of \,$T$\,, we get \;$\|\,T\,\| \,\geq\, \left\|\,T_{\,W}\,\right\|$\; and therefore \,$\|\,T\,\| \,=\, \left\|\,T_{\,W}\,\right\|$\,.\;Clearly \;$T\,(\,x \,,\, b\,) \,=\, T_{\,W}\,(\,x \,,\, b\,)$\; for \,$x \,\in\, W$\,.\;This proves the theorem\,.
\end{proof}

\begin{theorem}\label{th8}
Let \,$X$\, be a linear 2-normed space over the field \,$\mathbb{R}$\; and let \,$x_{\,0}$\, be an arbitrary non-zero element in \,$X$\,.\;Then there exists a bounded b-linear functional \,$T$\, defined on \;$X \,\times\, \left <\,b\,\right >$\; such that 
\[ \|\,T\,\| \,=\, 1\; \;\;\&\;\;  \;T\,(\,x_{\,0} \,,\, b\,) \,=\, \left\|\,x_{\,0} \,,\, b\,\right\|\;.\]
\end{theorem}

\begin{proof}
Consider the set \;$W \,=\, \{\,t\,x_{\,0}\, \;|\;$\,where \;$t$\; is a arbitrary real number \,$\}$\,.\;Then it is easy to prove that \,$W$\, is a subspace of \,$X$\,.\;Define  \,$T_{\,W} \,:\, W \,\times\, \left <\,b\,\right > \,\to\, \mathbb{R}$\;  by,
\[ T_{\,W}\,(\,x \,,\, b\,) \,=\, T_{\,W}\,(\,t\,x_{\,0} \,,\, b\,) \,=\, t\,\left\|\,x_{\,0} \,,\, b\,\right\| \;,\; t \,\in\, \mathbb{R} \;.\] 
Note that \,$T_{\,W}$\, is a b-linear functional on \;$W \,\times\, \left <\,b\,\right >$\; with the property that 
\[T_{\,W}\,(\,x_{\,0} \,,\, b\,) \,=\, \left\|\,x_{\,0} \,,\, b\,\right\|\;.\] 
Further, for any \,$x \,\in\, W$, we have 
\[\left|\,T_{\,W}\,(\,x \,,\, b\,)\,\right| \,=\, \left|\,T_{\,W}\,(\,t\,x_{\,0} \,,\, b\,)\,\right| \,=\, t\, \left\|\,x_{\,0} \,,\, b\,\right\| \,=\, \left\|\,t\,x_{\,0} \,,\, b\,\right\| \,=\, \|\,x \,,\, b\,\|\;.\] 
So, \,$ T_{\,W}$\; is bounded b-linear functional and \,$\left\|\,T_{\,W}\,\right\| \,=\, 1$\,.\;Now, according to the Theorem (\ref{th7})\,, there exists a bounded b-linear functional \,$T$\, defined on \;$X \,\times\, \left <\,b\,\right >$\; such that \,$T\,(\,x \,,\, b\,) \,=\, T_{\,W}\,(\,x \,,\, b\,) \; \;\forall\; x \,\in\, W$\; and \;$\|\,T\,\| \,=\, \left\|\,T_{\,W}\,\right\|$\,.\;Therefore, \,$T\,(\,x_{\,0} \,,\, b\,) \,=\, T_{\,W}\,(\,x_{\,0} \,,\, b\,) \,=\, \left\|\,x_{\,0} \,,\, b\,\right\|$\; and \;$\|\,T\,\| \,=\, 1$\,.\;This completes the proof\,.
\end{proof}

\begin{theorem}
Let \,$X$\, be a linear 2-normed space over the field \,$\mathbb{R}$\; and let $\,x \,\in\, X$\, and \,$X^{\,\ast}_{\,b}$\; is the Banach space of all bounded b-linear functionals defined on \;$X \,\times\, \left <\,b\,\right >$\,.\;Then 
\[ \|\,x \,,\, b\,\| \,=\, \sup\,\left\{\, \dfrac{|\,T\,(\,x \,,\, b\,)\,|}{\|\,T\,\|} \,:\, T \,\in\, X^{\,\ast}_{\,b} \;,\; T \;\neq\; 0 \,\right\} \;.\]
\end{theorem}
\begin{proof}
If \,$x\,=\,0$\,, there is nothing to prove\,.\;Let \;$x \,\neq\, 0$\; be any element in \,$X$\,.\;By the Theorem (\ref{th8}), there exists a \;$T_{\,1} \,\in\, X^{\,\ast}_{\,b}$\; such that\;$T_{\,1}\,(\,x \,,\, b\,) \,=\, \|\,x \,,\, b\,\|$\; and \;$\left\|\,T_{\,1}\,\right\| \,=\, 1$\,.\;Therefore,
\begin{equation}\label{eq6}
\sup\,\left\{\;\dfrac{|\,T\,(\,x \,,\, b\,)\,|}{\|\,T\,\|} \,:\, T \,\in\, X^{\,\ast}_{\,b} \;,\; T \;\neq\; 0 \,\right\} \,\geq\, \dfrac{|\,T_{\,1}\,(\,x \,,\, b\,)\,|}{\|\,T_{\,1}\,\|} \,=\, \|\,x \,,\, b\,\|\;.
\end{equation}
On the other hand, \,$|\,T\,(\,x \,,\, b\,)\,| \,\leq\, \|\,T\,\| \, \|\,x \,,\, b\,\| \; \;\;\forall\; T \,\in\, X^{\,\ast}_{\,b}$\; and we obtain 
\begin{equation}\label{eq7}
\sup\,\left\{\, \dfrac{|\,T\,(\,x \,,\, b\,)\,|}{\|\,T\,\|} \,:\,T \,\in\, X^{\,\ast}_{\,b} \;,\; T \;\neq\; 0 \,\right\} \;\leq\; \|\,x \,,\, b\,\|\;.
\end{equation}
From (\ref{eq6}) and (\ref{eq7}), we can write\,, 
\[ \|\,x \,,\, b\,\| \,=\, \sup\,\left\{\, \dfrac{|\,T\,(\,x \,,\, b\,)\,|}{\|\,T\,\|} \,:\, T \,\in\, X^{\,\ast}_{\,b} \;,\; T \;\neq\; 0 \,\right\}\;.\] 
This completes the proof\,.
\end{proof}

%=====================================
\section{Conclusions}
%=====================================

\smallskip\hspace{.6 cm}
Hahn-Banach theorem, Uniform boundedness principle, also known as Banach-Steinhaus theorem, open mapping theorem, closed graph theorem are most fundamental theorems and determining tools in functional analysis\,.\;In this paper, in the settings of linear 2-normed space, we have established necessary and sufficient condition for a linear operator to be closed in terms of its graph, different types of continuity for b-linear functionals and some characterizations of them and finally uniform boundedness principle and Hahn-Banach extension theorem for bounded b-linear functionals\,.\;Yet it remains to establish another few important concepts of functional analysis like, reflexivity of linear 2-normed space, Hahn-Banach separation theorem for bounded b-linear functionals etc.\;in the settings of linear 2-normed space\,.\;Also, these results can further be developed in linear n-normed space\,.

\end{document}